\documentclass[a4paper]{amsart}
\newtheorem{lemma}{Lemma}
\newtheorem{theorem}[lemma]{Theorem}
\theoremstyle{remark}
\newtheorem{remark}[lemma]{Remark}
\usepackage{xcolor}
\usepackage[colorlinks, citecolor = blue]{hyperref}
\usepackage{amsmath, amsthm, amsfonts, mathrsfs, amssymb}
\usepackage{enumitem}

\usepackage{orcidlink}

\title{A solution to Crouzeix's conjecture}

\author[Lorist]{Emiel Lorist \orcidlink{0000-0002-2045-6035}}
\address{Emiel Lorist \hfill\break\indent
Delft Institute of Applied Mathematics \hfill\break\indent
Delft University of Technology \hfill\break\indent
P.O. Box 5031 \hfill\break\indent
2600 GA Delft, The Netherlands}
\email{e.lorist@tudelft.nl}

\author[Schwenninger]{Felix L.~Schwenninger \orcidlink{0000-0002-2030-6504}}
\address{Felix L.~Schwenninger\hfill\break\indent 
Department of Applied Mathematics\hfill\break\indent
University of Twente \hfill\break\indent
P.O. Box 217 \hfill\break\indent
7500 AE Enschede, The Netherlands}
\email{f.l.schwenninger@utwente.nl}

\thanks{E. Lorist was financed by the Dutch Research Council (NWO) on the project ``The sparse revolution for stochastic partial differential equations'' with project number \href{https://doi.org/10.61686/ZGRMR99948}{VI.Veni.242.057}.
}


\begin{document}

\begin{abstract}
We provide a proof of Crouzeix's conjecture, which combines the tools developed previously for weaker estimates with a simple perturbation lemma for $2$-dilations. %The key point is that the lemma uses the uniformly bounded errors at all powers simultaneously, rather than relying only on the first-order relation.
Applying the lemma to the iterates $f^{n}$ in the double-layer potential representation yields the conjectured bound.
\end{abstract}

\maketitle
\section{Introduction}
In 2004, Crouzeix \cite{Crouzeix2004} conjectured that for any dimension $d\in\mathbb{N}$, any linear operator $A$ on $H=\mathbb{C}^{d}$ and any polynomial $p$ with complex coefficients,  
\begin{equation}\label{eq0}\|p(A)\|\leq 2\max_{z\in W(A)}|p(z)|,
\end{equation}
where $W(A)=\{\langle Ax,x\rangle\colon x\in H,\|x\|=1\}$ denotes the numerical range of $A$ and $\|\cdot\|$ denotes the operator norm induced by the Euclidean norm. Subsequently Crouzeix \cite{Crouzeix2007} showed  that  \eqref{eq0} holds if $2$ is replaced by $11.08$. Ten years later, Crouzeix and Palencia \cite{CrouzeixPalencia} showed that the constant can be lowered to $1+\sqrt{2}$. Later, it was shown in \cite{RansfordMalman} that $1+\sqrt{2}$ is never attained.  The conjecture has been shown for many special classes of matrices and operators in recent years \cite{bickelgorkin,CrouzeixGreenbaumLi,Glader,oloughlin2024crouzeix}. Let us mention that the case $d=2$ was already obtained in \cite{Crouzeix2004} and the case where the numerical range is a disk goes back to a classical result by Okubo and Ando on $2$-contractions \cite{OkuboAndo}.

The main technical tool used to derive results for general matrices (and operators) in most references above is the operator-valued  double-layer potential, which originates in the seminal work of  Delyon and Delyon \cite{DelyonDelyon}. For the main features of this technique, we refer to \cite{SchwenningerdeVries2}. The use of the double-layer potential  was refined in \cite{CrouzeixPalencia}, which also initiated a research program to abstract the function-theoretic ``hard analysis'' into operator-theoretic ``soft analysis'' %, as present in several areas of complex analysis, 
\cite{Bickeletal,ClouatreOstermannRansford,RansfordSchwenninger,SchwenningerdeVries1}. 
This note continues and completes this abstract approach by supplying the final piece of the puzzle.

We note in passing the efforts to use more ``hard analysis'' insights to resolve the conjecture, including the use of Blaschke products $B$, which after composition with a conformal mapping, realize the (unknown) optimal constant in \eqref{eq0}, see \cite{Bickeletal,CaldwellGreenbaum,ClouatreOstermannRansford}. In particular, we mention the more recent proof of the $d=2$ case in \cite{Mashreghietal}. 
\smallskip

 During the preparation of this manuscript, a proof of Crouzeix's conjecture appeared independently in \cite{Jin} using an approach involving function-theoretic representations of matrix-valued functions.
 Our proof instead uses the operator-theoretic Lemma~\ref{lem1}. 
\smallskip

In the rest of this note, all Hilbert spaces $H$ are complex, inner products are linear in the first variable, and $\mathcal{L}(H)$ denotes the  space of bounded linear operators on $H$.

\section{Main result}
The following result deals with perturbations of $2$-dilations with respect to a contraction. Indeed, if the perturbations are uniformly bounded over the powers of the operator and a commutativity condition holds, then the operator norm cannot exceed $2$. 

\begin{lemma}\label{lem1}
Let $H$ be a Hilbert space and $T\in \mathcal{L}(H)$. Suppose that there exists a Hilbert space $K$, a contraction $Q\in\mathcal{L}(K)$  and an isometry $V\in\mathcal{L}(H,K)$ such that the operators 
\begin{equation}\label{eq1}
E_{n}:=2V^*Q^{*n}V-T^{*n}, \qquad n\in \mathbb{N},
\end{equation}
are uniformly bounded and commute with $T$. Then $\|T\|\leq 2$.
\end{lemma}
\begin{proof} Define $\kappa:=\|T\|$. If $\kappa \leq 1$, there is nothing to prove, so suppose that $\kappa>1$. 
Without loss of generality we may assume that there exists an $x \in H$ such that $\|x\|=1$ and $T^{*}Tx=\kappa^{2}x$. This is clear when $H$ is finite-dimensional. For infinite-dimensional $H$, we can simultaneously pass to an ultrapower of $H$ and $K$, which preserves all assumptions of the lemma and the norm of $T$ and makes the operator $T^*T$ norm-attaining.

 Set $m_{n}:=\Re\langle E_{n}T^{n}x,x\rangle$. The powers $T^{n}$ are uniformly bounded by \eqref{eq1} and the uniform boundedness of the $E_n$'s, hence the sequence $(m_n)_{n\geq 1}$ is bounded. Moreover, since $T$ commutes with the $E_{n}$'s, we have for $n \in \mathbb{N}$
\begin{align*}
\kappa m_n-m_{n+1}
&=\Re\langle T^n(\kappa E_n-E_{n+1}T)x,x\rangle\\
&=\Re\langle T^n\bigl(\kappa(\kappa-1)T^{*n}x
   +2V^*Q^{*n}(\kappa Vx-Q^*VTx)\bigr),x\rangle,\\
   &=\kappa(\kappa-1)\|T^{*n}x\|^{2}+2 \Re\langle V^* Q^{*n}(\kappa Vx-Q^*VTx),T^{*n}x\rangle.
\end{align*}
Thus, defining the error vector $u:= Q^*VTx-\kappa Vx,$ we have
\begin{align*}
\kappa m_{n}-m_{n+1}={}&\kappa(\kappa-1)\|T^{*n}x\|^{2}-2\,\Re\langle V^*Q^{*n}u,T^{*n}x\rangle\\
={}&\kappa(\kappa-1)\|T^{*n}x-\tfrac{1}{\kappa(\kappa-1)}V^{*}Q^{*n}u\|^{2}-\frac{\|V^{*}Q^{*n}u\|^{2}}{\kappa(\kappa-1)}\geq -\frac{\|u\|^{2}}{\kappa(\kappa-1)}.%\\
%\geq{}& \kappa(\kappa-1) \|T^{*n}x\|^{2}-2\,\|u\|\|T^{*n}x\|\\
%={}&\kappa(\kappa-1) \Bigl( \|T^{*n}x\|-\frac{\|u\|}{\kappa(\kappa-1)}\Bigr)^2 -\frac{\|u\|^{2}}{\kappa(\kappa-1)}\geq -\frac{\|u\|^{2}}{\kappa(\kappa-1)}.
\end{align*}
Therefore, by a telescoping argument we have
\begin{equation}\label{eq2}
\kappa m_1=\sum_{n=1}^{\infty}\kappa^{-n+1}(\kappa m_{n}-m_{n+1}) \geq -\frac{\|u\|^{2}}{(\kappa-1)^2}.
\end{equation}
%Moreover, $\|u+\kappa Vx\|^2= \|Q^*VTx\|^2 \leq \kappa^2$, so $2\Re\langle{u,Vx}\rangle \leq -\frac{\|u\|^2}{\kappa}$ and therefore
%\begin{align}\label{eq:upper}\begin{aligned}
 %   m_1 &= 2\Re\langle Q^*VTx,Vx\rangle-\kappa^2 \\&= 2\Re\langle u,Vx \rangle +2\kappa-\kappa^2  
 %   \leq \kappa(2-\kappa)-\frac{\|u\|^2}{\kappa}.
 %   \end{aligned}
%\end{align}
Moreover, using again that $V$ is an isometry and $Q$ a contraction, 
 \begin{align}\begin{aligned}\|u\|^{2}={}&\|Q^{*}VTx\|^{2}+\kappa^{2}-2\kappa\Re\langle V^{*}Q^{*}VTx,x\rangle\\
 \leq{}& 2\kappa^{2}-\kappa\Re\langle E_{1}Tx,x\rangle -\kappa^{3}=2\kappa^{2}-\kappa m_{1}-\kappa^{3}\end{aligned}\label{eq4}
 \end{align}
Combined with \eqref{eq2}, we find
\begin{equation*}
   \|u\|^2\bigl(1-\tfrac{1}{(\kappa-1)^2}\bigr) \leq \kappa^2(2-\kappa),
\end{equation*}
which is impossible if $\kappa>2$. Therefore, $\kappa\leq 2$.
\end{proof}

\begin{remark}\label{rem1}~
\begin{enumerate}[label=(\roman*)] 
    \item \label{rem1it1}
The upper bound for $m_1=\Re\langle E_1Tx,x\rangle$
in \eqref{eq4} implies that
\[
\kappa\leq 1+\sqrt{1-\Re\langle E_1Tx,x\rangle}.
\]
Thus, if $\Re\langle E_1Tx,x\rangle\geq 0$, the assertion of Lemma \ref{lem1} follows directly. However, this positivity condition does not hold in general; see Section \ref{sec:epi}\ref{it:extremal}. The key new idea in the proof of the lemma is the lower bound in \eqref{eq2}, obtained by exploiting the identities \eqref{eq1} simultaneously for all $n\geq 1$. Combined with the upper bound \eqref{eq4}, it yields the conclusion of Lemma \ref{lem1} even when $\Re\langle E_1Tx,x\rangle<0$.
\item \label{rem1it2} If the dilation factor $2$ is replaced by a general $\rho>0$ in \eqref{eq1}, an analogous computation yields the bound 
\[\|T\|\leq \max\left\{\rho,1+\sqrt{\tfrac{\rho}{2}}\right\}.\]
\end{enumerate}
 \end{remark}
The following result shows that Crouzeix's conjecture is true. For the reader's convenience, we present a proof explicitly stating the machinery used in earlier works. This can be  shortened by an abstract argument, see Section \ref{sec:epi}.

\begin{theorem}[Crouzeix's conjecture]\label{thm}
Let $A$ be a bounded operator on a Hilbert space $H$. Then the numerical range $W(A)$  is a $2$-spectral set, i.e.\ 
\begin{equation}
\label{eq5}
\|f(A)\|\leq 2\sup_{z\in W(A)}|f(z)|
\end{equation}
for all rational functions $f$ with poles off the closure of $W(A)$.
\end{theorem}
\begin{proof}
It suffices to consider the case $H=\mathbb{C}^{d}$ for arbitrary $d\in\mathbb{N}$ and to replace $W(A)$ in the statement by any smoothly bounded, open convex set $\Omega$ containing $\overline{W(A)}$;  see \cite{SchwenningerdeVries1,SchwenningerdeVries2}.  We recall the relevant ingredients of the standard setup developed in \cite{Crouzeix2007,CrouzeixPalencia,DelyonDelyon}. See \cite{SchwenningerdeVries2} for details of the facts stated in the following paragraph. 

As the central object, we consider the double-layer potential
$P_{\Omega}:\partial\Omega\to\mathcal{L}(H)$ given by $$P_{\Omega}(\sigma):=\tfrac{1}{\pi}\Re\left(n_{\Omega}(\sigma)(\sigma I-A)^{-1}\right)$$ with $n_{\Omega}(\sigma)$ being  the outward unit normal to $\partial\Omega$ at $\sigma$. Let $\mathcal{A}(\Omega)$ denote the analytic functions $f:\Omega\to\mathbb{C}$ which extend continuously to the closure of $\Omega$ and define the mapping $\Phi:\mathcal{A}(\Omega)\to \mathcal{L}(H)$ by
\begin{equation}\label{eqthm1}
\Phi(f):=\frac{1}{2}\int_{\partial\Omega}f(\sigma)P_{\Omega}(\sigma)\left|\mathrm{d}\sigma\right|,\qquad f\in \mathcal{A}(\Omega).
\end{equation} 
Since $W(A)\subset\Omega$, it readily follows that $P_{\Omega}(\sigma)\ge0$ for every $\sigma\in\partial\Omega$ and, by the Cauchy integral theorem, $\Phi(1)=I$. It is also known that 
there exists a bounded, antilinear mapping $\alpha:\mathcal{A}(\Omega)\to\mathcal{A}(\Omega)$ such that  $$2\Phi(f)-f(A)=\alpha(f)(A)^{*}.$$ Indeed, although its precise form is irrelevant here, $\alpha(f)$ is the Cauchy transform of $\bar{f}$, i.e. $$\alpha(f)=\frac{1}{2\pi i}\int_{\partial\Omega} \overline{f(\sigma)}(\sigma-\cdot)^{-1}\mathrm{d}\sigma.$$ 

Let $K=L^{2}(\partial \Omega,|\mathrm{d}\sigma|; H)$ and let $V:H\to K$ be given by $Vx:= \frac{1}{\sqrt{2}}P_{\Omega}(\cdot)^{\frac{1}{2}}x$, which defines an isometry as $\Phi(1)=I$.  Fix $f\in \mathcal{A}(\Omega)$ with supremum norm $\|f\|_\infty=1$. Then $Q:K\to K$ given by $Q(g)=f\cdot g$ is contractive and 
\[V^{*}Q^{n}V=\frac{1}{2}\int_{\partial\Omega}f^{n}(\sigma)P_{\Omega}(\sigma)\,|\mathrm{d}\sigma|=\Phi(f^n), \qquad n\in \mathbb{N}.\]
By the above facts, 
\begin{equation}\label{eq:importantDL}
E_{n}:=2V^{*}Q^{*n}V-T^{*n}= 2\Phi(f^{n})^{*}-f(A)^{*n}=\alpha(f^{n})(A),
\end{equation}
whence it follows that $T:=f(A)$ commutes with $E_{n}$ 
%(A)=(f^{n}\alpha(f^{n}))(A)=\alpha(f^{n})(A)T$ 
by the homomorphism property of the holomorphic functional calculus $\theta:\mathcal{A}(\Omega)\to\mathcal{L}(H)$ given by $g\mapsto g(A)$. Finally, note that $$\|E_{n}\|=\|\alpha(f^{n})(A)\|\leq \|\theta\|\|\alpha\|\|f^{n}\|_\infty= \|\theta\|\|\alpha\|<\infty$$ for all $n\in\mathbb{N}$. Lemma \ref{lem1} now gives $\|T\|=\|f(A)\|\leq 2$.
% which coincides with the singular double-layer potential (defined on $C(\partial\Omega)$ on $\mathcal{A}(\Omega)$,
\end{proof}
\pagebreak[2]

\section{Remarks and complements}\label{sec:epi}
In this section, we collect several observations on Theorem \ref{thm} and its proof, and place the underlying ideas and techniques in their historical context.

\begin{enumerate}[label=(\roman*)]\setlength\itemsep{0.5em}
\item \textbf{Comparison of the ingredients used with earlier approaches.}\\ 
The proof of Theorem \ref{thm} does not use that $\alpha$ is contractive (as a consequence of the convexity of the domain), which was the key new ingredient used in Crouzeix--Palencia's  improvement \cite{CrouzeixPalencia} of the old estimate as well as in the alternative proof in \cite{RansfordSchwenninger}, see also \cite{ClouatreOstermannRansford,SchwenningerdeVries1}. In some sense, nearly all facts used in the  proof of Theorem \ref{thm} were already present in Crouzeix's original proof \cite{Crouzeix2007}, with the minor exception that he did not explicitly state that $2\Phi(f)-f(A)$ is given by $\alpha(f)(A)^{*}$. This identity is needed to show the commutativity assumption in Lemma \ref{lem1}. However, he did prove the estimate 
$\|2\Phi(f)-f(A)\|\leq 9.08\,\|f\|_\infty,$ which implies $\sup_{n\in\mathbb{N}}\|E_{n}\|<\infty$. Additionally, the use of $\Phi(1)=I$ to define the isometry $V$ is new.

\item \label{it:extremal}\textbf{Extremal functions and vectors.}\\
The central object in the proof of Lemma \ref{lem1}, translated to the proof in Theorem \ref{thm}, is 
\[m_{1}=\Re\langle E_{1}Tx,x\rangle=\Re\langle f(A)\alpha(f)(A)x,x\rangle=\Re\langle (f\alpha(f))(A)x,x\rangle,\]
which has already played a central role in previous approaches. Indeed, for an extremal pair
$(f,x)\in\mathcal{A}(\Omega)\times H$, that is, unit-norm elements $f$ and $x$  such that
$\|f(A)x\|=\sup_{h\in\mathcal{A}(\Omega),\|h\|_\infty\leq1}\|h(A)\|$,
this expression was studied in \cite{Bickeletal,SchwenningerdeVries1}. In \cite{SchwenningerdeVries1} it was shown that $|\langle (f\alpha(f))(A)x,x\rangle|\leq 1$ and thus $m_1\geq -1$. Using Remark \ref{rem1}\ref{rem1it1}, this yields the constant  $1+\sqrt{2}$.
Afterwards, in \cite{SchwenningerdeVries2} it was shown that $m_{1}\geq0$ implies that Crouzeix's conjecture is true, which is essentially Remark \ref{rem1}\ref{rem1it1}. This was also used in \cite{CrouzeixGreenbaumShift} to reprove Choi's result \cite{Choi} for weighted shift matrices.

Recently, it was shown in \cite{Mashreghietal} that $m_{1}\geq0$ holds in dimension $d=2$ if $\Omega = W(A)$ (strictly speaking, one has to replace $W(A)$ by its interior). However, examples of $3\times 3$ matrices can be constructed such that $m_{1}$ is strictly negative. Furthermore, for $d=2$ such examples  can also be constructed if the set $\Omega$	is chosen strictly larger than $W(A)$.
Thus, the fact that Lemma \ref{lem1} allows $m_{1}<0$ is not just an artifact of our proof. 

We also remark that our proof does not exploit any specific properties of the extremal functions or vectors \cite{Bickeletal,CaldwellGreenbaum}, or the related extremal measures \cite{SchwenningerdeVries1}.

\item \textbf{The refinement from $n=1$ to all $n\geq 1$.}\\
The identity \eqref{eq1} for $n=1$, combined with $E_1 = \alpha(f)(A)$ and the contractivity of $\alpha$, see \cite{CrouzeixPalencia,SchwenningerdeVries2},  also implies the weaker spectral constant $1+\sqrt{2}$. Indeed, as in \cite{RansfordSchwenninger}, \eqref{eq1} with $n=1$ gives 
\begin{equation*}
T^{*}TT^{*}T=T^{*}T\left(2V^{*}Q^* V\right)T-T^{*}TE_{1}T.
\end{equation*}
For an extremal $f$, with $T$, $E_{1}$ and $\kappa$ as in the proof of Theorem \ref{thm}, the first term on the right-hand side has norm at most $2\kappa^{3}$ and the second term satisfies
 \[\|T^{*}TE_{1}T\|\leq \|f(A)\|\|(f\alpha(f)f)(A)\|\leq \kappa \|\theta\|\|f\|_\infty^{2}\|\alpha(f)\|_\infty \leq \kappa^2.\]
This yields $\kappa^{4}\leq 2\kappa^{3}+\kappa^{2}$, and hence $\kappa\leq1+\sqrt{2}$. 

The improvement from $1+\sqrt{2}$ to $2$ in our argument comes from leveraging the identities \eqref{eq1} simultaneously for all $n\geq 1$, and hence from combining the algebraic structure of the functional calculus with the norm estimates in a more refined way. A related phenomenon occurs in Esterle's proof of the zero-two law for cosine functions \cite[Section 3]{Esterle}: there, an algebraic identity for cosine functions is combined with a refined norm inequality to upgrade a non-sharp estimate to the sharp bound. This argument may be viewed as a loose inspiration for the proof of Lemma \ref{lem1}.

\item \label{abstractlemma}\textbf{Extensions and limitations}\\
The proof of Theorem \ref{thm} can be adapted to show the following general abstract result, which settles a variant of  \cite[Question 4.1]{RansfordSchwenninger}, see also \cite[Conjecture 1.2 and Theorem 1.3]{ClouatreOstermannRansford}: 
    \medskip
    \begin{quote}
        Let $\mathcal{A}$ be a uniform algebra, $\alpha:\mathcal{A}\to\mathcal{A}$ a unital, antilinear, bounded map and $\theta:\mathcal{A}\to \mathcal{L}(H)$ a unital, bounded homomorphism. If $\Phi=\frac{1}{2}\left(\theta(\cdot)+\theta(\alpha(\cdot))^{*}\right)$ is completely positive, then $\|\theta\|\leq2$.
    \end{quote}
    \medskip
    To see this, one uses Arveson's extension and Stinespring's dilation theorem, see \cite{Paulsen}, to extend $\Phi$ to a $C^{*}$-algebra $\mathcal{C}$ and to deduce the existence of a Hilbert space $K$, an isometry $V:H\to K$ and a $*$-homomorphism $\pi:\mathcal{C}\to \mathcal{L}(K)$ such that 
    \[\Phi(f)=V^{*}\pi(f) V.\]
    For $f \in \mathcal{A}$ with $\|{f}\|=1$ set $T=\theta(f)$ and $Q=\pi(f)$, which settles the claim as in the proof of Theorem  \ref{thm}. Note that the $\Phi$ in the proof of Theorem \ref{thm} indeed extends to a positive, and thus completely positive, map on the commutative $C^{*}$-algebra $C(\partial\Omega)$ of the continuous functions on $\partial\Omega$.

    We point out that neither our proof nor Jin's proof \cite{Jin} can be used in a direct way to derive the completely bounded case of Crouzeix's conjecture. For this version of the conjecture and related recent results, see \cite{ClouatreOstermannRansford,HartzMcCarthy}. In our argument, the obstruction is reflected by the commutativity assumption in Lemma \ref{lem1}, which is not available after matrix amplification.
    \item \textbf{Extensions to other spectral sets}\\
    The abstract result from \ref{abstractlemma} above immediately implies the optimal spectral constant $2$ for the quantum annulus, see \cite{CrouzeixAnnulus,CrouzeixGreenbaum,Jin,JuryTsikalas,Pascoe,Tsikalas2022}.
    The generalization of Lemma \ref{lem1} described in Remark \ref{rem1}\ref{rem1it2} can be used to derive spectral constants for other (possibly non-convex) sets other than the numerical range, see  \cite{CrouzeixGreenbaum}. 
\end{enumerate}

\subsection*{AI disclosure statement}
   GPT-5.6 Sol Pro, accessed through ChatGPT by OpenAI, was used to explore proof strategies for this note. More precisely, it was used on July 25 and 26, 2026\footnote{The authors became aware of \cite{Jin} on July 28, 2026, after the interactions described above.}, to review previous approaches to the weaker spectral constant $1+\sqrt{2}$ and to identify a possible source of improvement in estimates involving iterates $f^{n}$ of extremal or approximately extremal functions. A key observation emerging from these interactions was that the discrepancy term
\[\bigl\|\Phi(f^{n+1})^{*}f(A)x-\|f(A)\|\Phi(f^{n})^{*}x\bigr\|,\]
arising from the double-layer-potential identity \eqref{eq:importantDL}, 
can be controlled uniformly in $n$ by the same quantity that naturally arises when estimating the previously studied remainder term $\Re\langle( f\alpha(f))(A)x,x\rangle$ for unit-norm $f$ and extremal $x$. 
This yielded a preliminary result, from which Lemma \ref{lem1} was subsequently abstracted. The proof in its final form was developed in full, verified, and written by the authors, who take full responsibility for all claims and arguments in the paper.

\section*{Acknowledgments}
The authors thank Jochen Gl\"uck and Michael~Hartz for valuable discussions.

%\bibliographystyle{abbrv}
%\bibliography{crouzeix}
\end{document}